\documentclass{amsart}

\usepackage{mathrsfs}
\usepackage{latexsym}
\usepackage{latexsym, url}
\usepackage{amsthm}
\usepackage{amsmath}
\usepackage{amsfonts}
\usepackage{amssymb}
\usepackage[dvips]{graphicx}
\usepackage[T1]{fontenc}
\usepackage[mathscr]{eucal}
\usepackage{xypic}
\input xy
\xyoption{all}
\usepackage{hyperref}

\newtheorem{theorem}{Theorem}[section]

\newtheorem{lemma}[theorem]{Lemma}

\theoremstyle{definition}
\newtheorem{definition}[theorem]{Definition}
\theoremstyle{remark}
\newtheorem{remark}[theorem]{Remark}
\newtheorem{example}[theorem]{Example}

\newcommand{\C}{\mathscr{C}}
\newcommand{\map}{\mathrm{map}}
\renewcommand{\L}{\mathcal{L}}
\newcommand{\R}{\mathcal{R}}
\renewcommand{\S}{\mathscr{S}}
\newcommand{\Ac}{\mathscr{A}}
\newcommand{\Coac}{\mathscr{B}}
\newcommand{\Z}{\mathbb{Z}}

\begin{document}
\title{Some characterizations of acyclic maps}
\author{George Raptis}
\address{\newline
G. Raptis \newline
Fakult\"{a}t f\"ur Mathematik, Universit\"{a}t Regensburg, D-93040 Regensburg, Germany}
\email{georgios.raptis@ur.de}

\begin{abstract}
We discuss two categorical characterizations of the class of acyclic maps between (path-connected) spaces. The first one 
is in terms of the higher categorical notion of an epimorphism. The second one employs the notion of a balanced map, that 
is, a map whose homotopy pullbacks define also homotopy pushouts. We also identify the modality in the homotopy theory of 
spaces that is defined by the class of acyclic maps, and discuss the content of the generalized Blakers-Massey theorem for this modality. 
\end{abstract}

\maketitle

\section{Introduction and Statement of Results}

One of the most enigmatic elementary constructions in homotopy theory is Quillen's plus construction. 
For each based path-connected space $X$ there is a based acyclic map $q_X \colon X \to X^+$ such that the 
kernel of $\pi_1(q_X)$ is exactly the maximal perfect subgroup of $\pi_1(X)$. Such a construction first 
appeared for homology spheres in \cite{Ker} and it was later developed by Quillen \cite{Qui} 
who used it for the purpose of defining higher algebraic $K$-theory of rings. The construction easily
generalizes as follows: for each perfect normal subgroup $P \unlhd \pi_1(X)$,  there is an acyclic map 
$q_{X, P} \colon X \to X^+_P$ such that the kernel of $\pi_1(q_{X, P})$ is $P$. The map $q_{X, P}$ is universal 
with respect to this property and therefore also unique up to homotopy.

\smallskip 

We recall that a map of path-connected spaces $f \colon X \to Y$ is \emph{acyclic} if its homotopy fiber 
$F_f$ is acyclic, i.e., $\widetilde{H}_*(F_f; \mathbb{Z}) = 0$. Acyclic maps are also characterized by the following 
propertes (see \cite{Ber, HH}):
\begin{itemize}
 \item[(H)] the map $f$ induces isomorphisms $H_*(X; f^*A) \cong H_*(Y; A)$ for all local abelian 
 coefficient systems $A$ on $Y$.
 \item[(+)] the map $f$ is identified up to weak homotopy equivalence with the plus construction $q_{X, P} \colon X \to X^+_P$
 with respect to the kernel $P$ of $\pi_1(f,x)$ -- which is a perfect group.
\end{itemize}
More generally, we say that a map $f \colon X \to Y$ of (not necessarily path-connected) spaces is \emph{acyclic} if the homotopy fibers of $f$ are acyclic spaces (equivalently, 
if (H) is satisfied). An acyclic map is therefore a homology equivalence in a strong sense, but it can be very far from inducing a $\pi_*$-isomorphism 
in general. For example, the theorem of Kan and Thurston \cite{KT} shows that for any path-connected space $X$, 
there is a discrete group $G$ and an acyclic map $BG \to X$. Note that acyclic maps are closed under homotopy pullbacks 
(by definition) and under homotopy pushouts (by (H)). Property (+) connects the acyclic maps with the plus construction 
and therefore with algebraic $K$-theory.  We refer to \cite{Ber} for a nice treatment of the approach to algebraic $K$-theory based on the plus construction.  

\medskip

This paper makes no special claim to originality. Its purpose is twofold, firstly, to prove two further characterizations of the class of acyclic maps, and secondly, 
to discuss the properties of the modality in the homotopy theory of spaces that is defined by the class of acyclic maps. The first characterization 
involves the notion of an epimorphism as suggested by higher category 
theory. In the homotopy theory ($\infty$-category) of spaces, this notion corresponds to the following: a map $f \colon X \to Y$ is called a (\emph{homotopy}) \emph{epimorphism} if the 
commutative square 
\begin{equation} \label{epi-pushout}
\xymatrix{
X \ar[d]_f \ar[r]^f & Y \ar[d]^{\text{id}} \\
Y \ar[r]^{\text{id}} & Y 
}
\end{equation}
is a homotopy pushout. This is the categorical dual of the notion of monomorphism as defined in higher category theory.
(See \cite{Lur} for a discussion of this notion in general $\infty$-categories.) The first characterization states that 
a map is acyclic if and only if it is an epimorphism. 

\smallskip

The second characterization is motivated by another well known special property of acyclic maps. Given an acyclic 
map  between path-connected spaces $f \colon X \to Y$, then the homotopy fiber sequence 
\begin{equation} \label{fiber=cofiber}
F_f \to X \xrightarrow{f} Y
\end{equation}
is also a homotopy cofiber sequence (see \cite{HH}). Indeed, the induced map from the homotopy cofiber $C_{F_f \to X} \to Y$ is acyclic 
(by excision) and induces a $\pi_1$-isomorphism (by the van Kampen theorem). However, it is easy to see that this property cannot characterize 
the acyclic maps (the trivial map $X \to \ast$ always has this property, too). We consider the following stronger property: a map between path-connected 
spaces $f \colon X \to Y$ is called \emph{balanced} if for every map $g \colon B \to Y$ where $B$ is non-empty, the homotopy pullback square 
\begin{equation} \label{balanced-square}
\xymatrix{
E \ar[d] \ar[r] & X \ar[d]^{f} \\
B \ar[r]^{g} & Y 
}
\end{equation}
is also a homotopy pushout. The second characterization states that a map of path-connected spaces is acyclic if and only if it is balanced. This characterization 
follows easily from the results of Alonso in \cite{Alo} -- where general criteria for a homotopy pullback of (path-connected) spaces to be a homotopy pushout are shown. 
We prove these two characterizations of acyclic maps, as epimorphisms and as balanced maps, in Section 2. 

\medskip

The notion of a \emph{modality} was used by Anel-Biedermann-Finster-Joyal \cite{ABFJ} in connection with generalizations of the classical Blakers-Massey theorem. The prototypical example of a modality 
is the factorization system in the $\infty$-topos of spaces which is defined by the classes of $n$-connected and $n$-truncated maps. It is shown in \cite{ABFJ} that there is a 
Blakers-Massey theorem and a dual Blakers-Massey theorem associated to each modality. These theorems specialize to the classical statements in the case of the aforementioned 
modality in spaces. 

In Section 3, we give a proof that the class of acyclic maps defines the left class of a modality in spaces and identify the associated right class as the class of maps 
$f \colon X \to Y$ such that $\mathrm{ker}\big(\pi_1(f,x)\big)$ does not contain a non-trivial perfect subgroup for each $x \in X$. We also apply the main results of \cite{ABFJ} 
and comment on the content of the corresponding Blakers-Massey theorems for this example of a modality. 

\medskip

\noindent \emph{Acknowledgements.} I thank J\'er\^{o}me Scherer for pointing 
out to me the work of Alonso \cite{Alo}, and Georg Biedermann for suggesting 
the connection with the notion of a modality and the results of \cite{ABFJ}. I also thank Mathieu Anel and Eric Finster 
for interesting discussions about the properties of modalities. This work is partially supported by \emph{SFB 1085 - Higher Invariants} 
(University of Regensburg) funded by the DFG.

\section{Epimorphisms, acyclic maps, and balanced maps}

In this section we prove two characterizations of acyclic maps, as epimorphisms in the homotopy theory of spaces, and as balanced maps.
The characterization of acyclic maps as balanced maps is essentially contained in the work of Alonso \cite{Alo} and our proof is closely 
related to arguments used in that paper (see \cite[Theorem 2.5 and Corollary 2.9]{Alo}). See also \cite[Proposition 4.2]{Alo} for some 
further characterizations of acyclic maps. 

\begin{theorem} \label{main}
Let $f \colon X \to Y$ be a map of spaces.
\begin{enumerate}
\item[(a)] $f$ is acyclic if and only if $f$ is an epimorphism.
\item[(b)] Suppose that $X$ and $Y$ are path-connected. Then $f$ is acyclic if and only if $f$ is balanced.
\end{enumerate}
\end{theorem}
\begin{proof} (a). Suppose that $f$ is an epimorphism. By the Mayer-Vietoris sequence 
of the homotopy pushout square \eqref{epi-pushout}, it follows that $f$ satisfies condition (H), 
and therefore it is acyclic. 

\medskip

\noindent Suppose that $f$ is acyclic. By restricting to path-connected components if necessary, we may 
assume that both $X$ and $Y$ are path-connected. Let $C$ be the homotopy pushout
$$
\xymatrix{
X \ar[d]_f \ar[r]^f & Y \ar[d] \\
Y \ar[r]_c & C 
}
$$
By excision, the map $c$ induces homology isomorphisms for all local abelian coefficient systems on $C$, and therefore $c$ is acyclic 
by condition (H). It is also a $\pi_1$-isomorphism by the van Kampen theorem. Therefore $c$ is a weak homotopy equivalence.

\medskip

\noindent (b). Suppose that $f$ is acyclic. Consider the following homotopy pullbacks 
$$
\xymatrix{
F \ar[d] \ar[r] & E \ar[d]^{f'} \ar[r] & X \ar[d]^{f} \\
\ast \ar[r] & B \ar[r] & Y 
}
$$
where $B$ is non-empty and path-connected. $F$ is the homotopy fiber of both $f$ and $f'$. Since $f$ and $f'$ are acyclic, the 
left and the composite squares are also homotopy pushouts - see the remarks after \eqref{fiber=cofiber} above. Therefore the right square is also 
a homotopy pushout.  This shows the claim in the case where $B$ is path-connected.

Let $B = \bigsqcup_I B_i$ where each $B_i$ is a path-connected component of $B \neq \varnothing$ (we may assume that $B$ is a CW complex). Then 
each homotopy pullback 
$$
\xymatrix{
E_i \ar[d] \ar[r]^{g'_i} & X \ar[d]^{f} \\
B_i \ar[r]^{g_i} & Y 
}
$$
is also a homotopy pushout since $f$ is acyclic and $B_i$ is path-connected.  Consider the diagram
$$
\xymatrix{
\bigsqcup_I E_i  \ar[d] \ar[rr]^{\bigsqcup_I g'_i} && \bigsqcup_I X  \ar[d]^{\bigsqcup_I f} \ar[r]^{\nabla} & X \ar[d]^f \\
B = \bigsqcup_I B_i \ar[rr]_{\bigsqcup_I g_i} && \bigsqcup_I Y \ar[r]_{\nabla} & Y
}
$$
By the definition of homotopy pullbacks, the composite square is a homotopy pullback. (This is an instance of homotopy descent in an $\infty$-topos \cite{Lur}.)

The left square is a homotopy pushout because it is so for each $i\in I$. Moreover, the right square is a homotopy pushout using the 
fact that $f$ is acyclic (or an epimorphism): the map $f$ is up to homotopy the composition of taking $|I| (>0)$ iterated homotopy pushouts 
along copies of the map $f$. Therefore the composite square is also a homotopy pushout and the claim in the general case follows. 

\medskip

\noindent Suppose that $f$ is balanced. Given a homotopy pullback 
\begin{equation} \label{pullback1}
\xymatrix{
E \ar[d]_{f'} \ar[r]^{g'} & X \ar[d]^{f} \\
B \ar[r]_g & Y 
}
\end{equation}
consider the square of homotopy fibers of the maps which start from each corner of the square \eqref{pullback1} and 
map to $Y$:
\begin{equation} \label{pullback2}
\xymatrix{
F_{gf'} \ar[d] \ar[r] & F_f \ar[d] \\
F_{g} \ar[r] & \ast
}
\end{equation}
This square is again a homotopy pullback by definition. Therefore $F_{gf'} \simeq F_{g} \times F_f$. 

Assuming that $B$ is non-empty, the first square \eqref{pullback1} is a homotopy pushout since $f$ is balanced. It follows by well known properties of the homotopy theory of spaces that the square of homotopy 
fibers \eqref{pullback2} is also a homotopy pushout. (This is an instance of homotopy descent in an $\infty$-topos \cite{Lur}.) Thus, the homotopy pushout 
$F_{g} \ast F_f$ of \eqref{pullback2} is weakly contractible for any $g \colon B \to Y$ with $B \neq \varnothing$. Setting, for example, $g \colon B = Y \times S^0 \to Y$ to be the projection, it follows that $F_f$ must have 
trivial reduced homology. This completes the proof of Theorem \ref{main}.  \end{proof}

\begin{remark}(On the connectivity assumption.)
For the obvious extension of the definition of a balanced map to non-path-connected spaces, we still have that a balanced map between (not necessarily path-connected) spaces is 
acyclic -- with the same proof as above applied to each homotopy fiber of $f$. But the converse fails for obvious reasons, even when the space $B$ is path-connected. An acyclic map 
between \emph{non}-path-connected spaces is balanced if and only if it is a weak homotopy equivalence.
\end{remark}

\begin{remark}
The $\infty$-category of spaces is well-copowered since quotient objects (in the sense of the notion of epimorphism used here) 
correspond to perfect normal subgroups of the fundamental group (at different basepoints).
\end{remark}

\begin{remark}
An acyclic map $f \colon X \to Y$ does \emph{not} have the property that each homotopy pushout square 
$$
\xymatrix{
X \ar[d]_f \ar[r] & C \ar[d] \\
Y \ar[r] & D
}
$$
is also a homotopy pullback. For example, given an acyclic space $F$, the homotopy pushout 
$$
\xymatrix{
F \ar[d] \ar[r] & \ast \ar[d] \\
\ast \ar[r] & \Sigma F \simeq \ast
}
$$
is not a homotopy pullback. We will return to this dual question in the next section.
\end{remark}

\begin{example} (Algebraic $K$-theory.)
Let $R$ be a unital ring and $BGL(R) \to BGL(R)^+$ the acyclic map to the path-connected cover of the algebraic $K$-theory of $R$. Then for each map 
$\varnothing \neq B \to BGL(R)^+$ the homotopy pullback 
$$
\xymatrix{
E \ar[d] \ar[r] & BGL(R) \ar[d] \\
B \ar[r] & BGL(R)^+ 
}
$$
is a homotopy pushout. For $B = \ast$, a model for the homotopy fiber $E$ is given by the (acyclic) \emph{Volodin space}. When $B = S^n$, then the 
homotopy pullback $E$ satisfies $E^+ \simeq S^n$. 
\end{example}

\begin{example}(Using the Kan-Thurston theorem.) Let $X$ be a path-connected space with universal covering space $p\colon \widetilde{X} \to X$. According to the Kan-Thurston theorem 
\cite{KT}, there is a discrete group $G$ and an acyclic (based) map $f \colon BG \to X$. We have a homotopy pullback square 
\begin{equation} \label{pullback3}
\xymatrix{
BP \ar[d] \ar[r] & BG \ar[d]^f \\
\widetilde{X} \ar[r]_p & X
}
\end{equation}
where $P$ is the kernel of $\pi_1(f)$ - a perfect normal subgroup of $G$. Then, by Theorem \ref{main}, the square \eqref{pullback3} is also a homotopy pushout. 
\end{example}

\section{The modality of acyclic maps}

\subsection{Preliminaries} For the presentation of the results in this section, it will be convenient to use the language of $\infty$-categories as in \cite{ABFJ}. Let $\S$ denote the $\infty$-category of spaces. Given $i \colon A \to B$ and $p \colon X \to Y$ in an $\infty$-category $\C$, we say that $i$ and $p$ are \emph{orthogonal} if 
the following square 
\[
\xymatrix{
\map_{\C}(B, X) \ar[r] \ar[d] & \map_{\C}(B, Y) \ar[d] \\
\map_{\C}(A, X) \ar[r] & \map_{\C}(A, Y)
}
\]
is a pullback in $\S$. In this case, we say that $i$ is \emph{left orthogonal} to $p$, or $p$ is \emph{right orthogonal} to $i$, and denote this relation by $i \perp p$. For 
a class of morphisms $A$ in an $\infty$-category $\C$, we write 
$$A^{\perp} \colon = \{p \in \C \ |\ a \perp p \ \textrm{for every} \ a \in A\}$$
$${}^{\perp}A \colon = \{i \in \C \ | \ i \perp a \ \textrm{for every} \ a \in A\}.$$

\begin{definition}
Let $\C$ be an $\infty$-category. A \emph{factorization system} in $\C$ consists of a pair of classes of morphisms $(\L, \R)$ in $\C$ such that:
\begin{itemize}
\item[(a)] every morphism $f$ in $\C$ admits a factorization $f = \R(f) \circ \L(f)$ where $\L(f) \in \L$ and $\R(f) \in \R$. 
\item[(b)] $\L^{\perp} = \R$ and $\L = {}^{\perp}\R$. 
\end{itemize}
We say that $\L$ is the \emph{left class} and $\R$ is the \emph{right class} of the 
factorization system. 
\end{definition}

\begin{definition} Let $\C$ be an $\infty$-topos. A factorization system $(\L, \R)$ in an $\infty$-topos $\C$ is a \emph{modality} if the class $\L$ is closed under pullbacks. 
\end{definition}

We refer to \cite[Section 3]{ABFJ} for a discussion of the general properties of factorization systems and modalities.

\subsection{Acyclic maps define a modality} Our purpose in this section is to give a proof that the class of acyclic maps in the $\infty$-topos $\S$ is the left class of a modality.
Let $\Ac$ denote the class of acyclic maps in $\S$. We recall that $f \colon X \to Y$ is acyclic if its homotopy fibers are acyclic spaces (equivalently, $f$ induces homology 
isomorphisms for all local abelian coefficient systems on $Y$). Furthermore, let $\Coac$ denote the class of maps $f \colon X \to Y$ in $\S$ such that for each 
$x \in X$ (= $x \colon 1 \to X)$, the normal subgroup 
$$\mathrm{ker}\big(\pi_1(f, x)\big) \unlhd\pi_1(X, x)$$
does \emph{not} contain a non-trivial perfect subgroup. Equivalently, $f \in \Coac$ if and only if for each $x \in X$ the (homotopy) fiber $F_{f,y}$ of $f$ at $y = f(x)$ has the property that  
its fundamental group $\pi_1(F_{f, y}, x)$ has trivial maximal perfect subgroup. (This uses the fact that the epimorphism $\pi_1(F_{f,y}, x) \twoheadrightarrow \mathrm{ker}\big(\pi_1(f,x)\big)$ defines a central extension, 
and therefore it preserves the maximal perfect subgroup \cite{Ber}.)

\begin{theorem} \label{main2}
The pair $(\Ac,\Coac)$ is a modality in $\S$. 
\end{theorem}

We will need the following lemma. We first recall that a map $f \colon X \to Y$ in $\S$ is a \emph{monomorphism} if it is equivalent to an inclusion of components (equivalently, if the canonical map $\Delta f \colon X \to X \times_Y X$ is an equivalence). 

\begin{lemma} \label{auxiliary}
Let $X$ be a $0$-connected space, $P \unlhd \pi_1(X,x)$ a perfect normal subgroup,
and let $q_{X, P} \colon X \to X^+_P$ denote the associated plus construction. Then the restriction map 
$$q_{X,P}^* \colon \map_{\S}(X^+_P, Z) \rightarrow \map_{\S}(X, Z)$$
is a monomorphism in $\S$ for every $Z \in \S$. Moreover, 
$$q_{X,P}^* \colon \pi_0\big( \map_{\S}(X^+_P, Z)\big) \hookrightarrow \pi_0\big(\map_{\S}(X, Z)\big)$$
is identified with the inclusion of the classes of maps $g \colon X \to Z$ for which $\mathrm{ker}\big(\pi_1(g,x)\big)$ contains the subgroup $P$.  
\end{lemma}
\begin{proof}
Since $X \to X^+_P$ is an epimorphism in $\S$, by Theorem \ref{main}(a), it follows that the square in $\S$
$$
\xymatrix{
\map_{\S}(X^+_P, Z) \ar[r] \ar[d] & \map_{\S}(X^+_P, Z) \ar[d] \\
\map_{\S}(X^+_P, Z) \ar[r] & \map_{\S}(X, Z)
}
$$
is a pullback. Therefore the first claim follows. The second claim is a well known consequence of the universal property of the plus construction (see, for example, \cite{Ber}).  
\end{proof} 

\medskip 

\noindent \textbf{Proof of Theorem \ref{main2}.}
First we show that the required factorizations exist. Let $f \colon X \to Y$ be a map of spaces. By restricting to components if necessary, we may assume that 
$X$ is $0$-connected, and we fix a basepoint $x \in X$. Let $P$ denote the maximal perfect subgroup of $\mathrm{ker}\big(\pi_1(f, x)\big)$. Using the universal property 
of the plus construction with respect to $P$, as a perfect normal subgroup of $\pi_1(X, x)$, we have a factorization in $\S$ as follows, 
\[
\xymatrix{
X \ar[dr]_i \ar[rr]^f && Y \\
& X^+_P \ar[ru]_p & 
}
\]
where $i \in \Ac$. Moreover, $p \in \Coac$ because there is an isomorphism 
$$\mathrm{ker}\big(\pi_1(p, i(x))\big) \cong \mathrm{ker}\big(\pi_1(f, x)\big)/P$$
and this group has no non-trivial perfect subgroups.  

Next we show that $\Coac \subseteq \Ac^{\perp}$ (or $\Ac \subseteq {}^{\perp} \Coac$). Consider a square in $\S$ 
$$
\xymatrix{
A \ar[r] \ar[d]_i & X \ar[d]^p \\
B \ar[r] & Y
}
$$
where $i \in \Ac$ and $p \in \Coac$. By restricting to components if necessary, we may assume that $i$ is a map between $0$-connected spaces. In this case, 
$B \simeq A^+_P$ for some perfect normal subgroup $P \unlhd \pi_1(A, a)$, for some chosen basepoint $a \in A$. We obtain the following diagram in $\S$
$$
\xymatrix{
\map(B, X) \ar[r]^(.6)c \ar[rd] & E \ar[r]^(.4){j} \ar[d] & \map(A, X) \ar[d] \\
& \map(B, Y) \ar[r]_{i^*} & \map(A, Y)
}
$$
where $E$ denotes the pullback. By Lemma \ref{auxiliary}, the maps $i^*$ and $j c$ are monomorphisms. Hence $j$, and as a consequence $c$, are also monomorphisms. Therefore it suffices to show that $c$ induces a bijection on $\pi_0$. By Lemma \ref{auxiliary}, we know that the subset
$$\pi_0(E) \hookrightarrow \pi_0 \big(\map(A, X)\big)$$ 
consists of those classes $g \colon A \to X$ such that $P$ is contained in $\mathrm{ker}\big(\pi_1(pg, a)\big)$. Since $p \in \Coac$, the image of the subgroup $P$ in $\pi_1(X, g(a))$ must be trivial, hence $P$ is actually 
contained in $\mathrm{ker}\big(\pi_1(g,a)\big)$. This identifies it with the 
subset $$\pi_0 \big(\map(B, X)\big) \hookrightarrow \pi_0 \big(\map(A, X)\big),$$ as was required to show. 

Now we prove that $\Ac^{\perp} \subseteq \Coac$. Suppose that $p \colon X \to Y$ is  in $\Ac^{\perp}$. Consider the factorization constructed above, $X \xrightarrow{i} \widetilde{X} \xrightarrow{q} Y$,
where $i \in \Ac$ and $q \in \Coac$. Then the square 
\[
\xymatrix{
X \ar@{=}[r] \ar[d]_i & X \ar[d]^p \\
\widetilde{X} \ar[r]_q \ar@{-->}[ru] & Y
}
\]
shows that $i$ admits a retraction, and therefore $i$ is an equivalence -- alternatively, note that $p$ is a retract of $q$. Lastly, we show that ${}^{\perp} \Coac \subseteq \Ac$. Suppose that $i\colon A \to  B$ is a map in ${}^{\perp} \Coac$. We consider again the factorization constructed above, $A \xrightarrow{j} \widetilde{B} \xrightarrow{p} B,$
where $j \in \Ac$ and $p \in \Coac$, and the lifting problem 
\[
\xymatrix{
A \ar[r]^j \ar[d]_i & \widetilde{B} \ar[d]^p \\
B \ar@{=}[r] \ar@{-->}[ru]^l & B.
}
\]
The lift $l$ shows that $i$ is a retract of the map $j$, hence it is also an acyclic map -- it also follows that $l$ must be an equivalence.  This completes the proof that $(\Ac, \Coac)$ is a factorization system. 
The factorization system $(\Ac, \Coac)$ defines a modality because acyclic maps are closed under pullbacks by definition. 
\qed

\medskip 

\begin{remark}
The modality of Theorem \ref{main2} can be regarded as an instance of the construction in \cite[Example 3.5.3]{ABFJ} for the nullification that defines the plus construction 
(see \cite{Tai, BC}). We emphasize that the modality of Theorem \ref{main2} is different from the factorization system that arises from the plus construction 
as a localization in the $\infty$-category of spaces and whose left class is the class of maps which become equivalences after plus construction. 
\end{remark}

\subsection{Blakers-Massey theorems for acyclic maps} As an application of the main results of \cite{ABFJ}, we obtain a Blakers-Massey theorem and a dual
Blakers-Massey theorem associated with the modality $(\Ac, \Coac)$ in $\S$. 

Given a map $f \colon X \to Y$ and $y \in Y$, we write $F_{f, y}$ for the (homotopy) fiber 
of $f$ at $y$. The Generalized Blakers-Massey theorem \cite[Theorem 4.1.1]{ABFJ} specialized to the modality $(\Ac, \Coac)$ gives the following statement (see also the comments in \cite[p. 30]{ABFJ}): Given a pushout square in $\S$
\[
\xymatrix{
A \ar[r]^g \ar[d]_f & C \ar[d] \\
B \ar[r] & D
}
\]
such that for every $a \in A$, the map 
\begin{equation} \label{auxiliary-2} 
F_{f, f(a)} \vee_a  F_{g, g(a)} \rightarrow F_{f, f(a)} \times F_{g, g(a)}
\end{equation} 
is acyclic, then the canonical map $A \to B \times_D C$ is acyclic.

\smallskip 

This statement can be reduced to a simpler statement as follows. First note that the 
fiber of the map \eqref{auxiliary-2} at the point $(a', a'')$ is given by the 
join 
$$\Omega_{a, a'} F_{f, f(a)} \ast \Omega_{a, a''} F_{g, g(a)}.$$
(Here $\Omega_{x,x'} X$, for $x ,x' \in X$, denotes the pullback of ($\ast \xrightarrow{x} X \xleftarrow{x'} \ast$).) Then it follows easily from the assumption that at most one of the fibers, $F_{f, f(a)}$ and $F_{g, g(a)}$, can fail to be $0$-connected, in which case the other fiber must be contractible -- and so the join is contractible as well. If both fibers are $0$-connected, then the following lemma applies.  

\begin{lemma}
Let $X$ and $Y$ be non-empty spaces in $\S$. Then $X \ast Y$ is contractible if and only if it is acyclic.
\end{lemma} 
\begin{proof}
Suppose that $X \ast Y$ is acyclic. Since $\widetilde{H}_0(X; \Z) \otimes \widetilde{H}_0(Y; \Z) \cong H_1(X \ast Y; \Z) = 0,$ 
it follows that either $X$ or $Y$ is $0$-connected. Then, by the van Kampen theorem applied iteratively, $X \ast Y$ is also $1$-connected (see also \cite[Lemma 2.2]{Mil}). Hence $X \ast Y$ is contractible. 
\end{proof}

Thus, the Blakers-Massey theorem for acyclic maps as stated above, reduces to the 
Little Blakers-Massey theorem \cite[Corollary 4.1.4]{ABFJ} specialized to $\S$: this asserts that if \eqref{auxiliary-2} is an equivalence for each $a \in A$, then the 
pushout square is also a pullback. 
  
\begin{remark}
The join $X \ast Y$ of two (non-empty) spaces is acyclic if and only if there is a set of primes $P$ such that one of the homologies, $\widetilde{H}_*(X; \Z)$ and $\widetilde{H}_*(Y; \Z)$, 
is $P$-torsion and the other is uniquely $P$-divisible. See \cite[Theorem 2.5]{Alo}.
\end{remark}

\begin{example}
Let $(X,x)$ and $(Y, y)$ be based $0$-connected spaces such that the join $\Omega_x X \ast \Omega_y Y$ is acyclic/contractible.  Then an application of the Little Blakers-Massey theorem to the pushout square in $\S$
\[
\xymatrix{
X \times Y \ar[r] \ar[d] & X \ar[d] \\
Y \ar[r] & X \ast Y
}
\]
shows that the square is also a pullback. (Further comments: the Blakers-Massey theorem for acyclic maps yields that the canonical map $X \times Y \to X \times_{X \ast Y} Y$ is acyclic. But either $\Omega_x X$ or $\Omega_y Y$ must be $0$-connected. As a consequence, the canonical map to the pullback is $\pi_1$-injective and therefore it is an equivalence.)
\end{example}

\begin{example}
Let $(X, x)$ and $(Y,y)$ be based spaces such that $\pi_1(X, x)$ and $\pi_1(Y,y)$ are trivial, and the join $\Omega^2_x X \ast \Omega^2_y Y$ is acyclic/contractible. Then an application of the Little 
Blakers-Massey theorem to the pushout square 
\[
\xymatrix{
\ast \ar[r] \ar[d] & X \ar[d] \\
Y \ar[r] & X \vee Y
}
\]
shows that it is also a pullback. 
\end{example}

On the other hand, the dual Generalized Blakers-Massey theorem \cite[Theorem 3.6.1]{ABFJ} for the modality $(\Ac, \Coac)$ specializes to the following statement: Given a pullback square in $\S$
\begin{equation} \label{dual-BM}
\xymatrix{
A \ar[r] \ar[d] & C \ar[d]^f \\
B \ar[r]_g & D
}
\end{equation}
such that that for every $d, d' \in D$, the join $F_{f, d} \ast F_{g, d'}$ is acyclic, then it follows that the canonical map $B \cup_A C \to D$ is acyclic.

\smallskip 

\noindent For the proof, note that the fiber 
of this canonical map at $d \in D$ is exactly the join $F_{f, d} \ast F_{g, d}$ -- this observation was also used in the proof of Theorem \ref{main}(b). Note that the assumption on the join is satisfied 
when $f$ or $g$ is acyclic. 

\smallskip

In the case of $0$-connected spaces, this statement is also part of \cite[Theorem 2.5]{Alo} in which case it is shown that $B \cup_A C \to D$ is actually an equivalence. The same conclusion 
holds more generally when the maps $f$ and $g$ are $\pi_0$-surjective in \eqref{dual-BM}, so that their fibers are non-empty -- this is also an instance of the dual Little Blakers-Massey theorem. On the other hand, the pullback square 
$$
\xymatrix{
\varnothing \ar[r] \ar[d] & C \ar[d]^f \\
\varnothing \ar[r] & D
}
$$
where $f$ is acyclic shows that the conclusion about acyclicity cannot be improved in general.

\end{document}